\documentclass{amsart}

\newcommand{\IR}{\mathbb R}
\newcommand{\IZ}{\mathbb Z}
\newcommand{\IN}{\mathbb N}
\newcommand{\w}{\omega}

\newtheorem{theorem}{Theorem}
\newtheorem{proposition}{Proposition}

\begin{document}
\title{An example of a non-Borel locally-connected finite-dimensional topological group}
\author{I.~Banakh, T.~Banakh, M.~Vovk}
\address{T.Banakh: Ivan Franko National University of Lviv}
\email{t.o.banakh@gmail.com}
\address{I.Banakh: Ya. Pidstryhach Institute for Applied Problems of Mechanics and Mathematics, Lviv}
\email{ibanakh@yahoo.com}
\address{M.Vovk: Lviv Polytechnic National University}
\email{mira.i.kopych@gmail.com}

\begin{abstract}
Answering a question posed by S.Maillot in MathOverFlow, for every $n\in\IN$ we construct a locally connected subgroup $G\subset\IR^{n+1}$ of dimension $\dim(G)=n$, which is not locally compact.
\end{abstract}
\maketitle

By a classical result of Gleason \cite{Gl} and Montgomery \cite{Mon}, every locally path-connected finite-dimensional topological group $G$ is locally compact. Generalizing this result of Gleason and Montgomery, Banakh and Zdomskyy \cite{BZ} proved that a topological group $G$ is locally compact if $G$ is compactly finite-dimensional and locally continuum-connected. In \cite{MO} Sylvain Maillot asked if the locally path-connectedness in the result of Gleason and Montgomery can be replaced by the local connectedness. In this paper we construct a counterexample to this question of S.~Maillot.
We recall that a subset $B$ of a Polish space $X$ is called a {\em Bernstein set} in $X$ if both $B$ and $X\setminus B$ meet every uncountable closed subset $F$ of $X$. Bernstein sets in Polish space can be easily constructed by transfinite induction, see \cite[8.24]{Ke}.

\begin{theorem} For every $n\ge n$ the Euclidean space $\IR^n$ contains a dense additive subgroup $G\subset\IR^n$
such that
\begin{enumerate}
\item $G$ is a Bernstein set in $\IR^n$;
\item $G$ is locally connected;
\item $G$ has dimension $\dim(G)=n-1$;
\item $G$ is not Borel and hence not locally compact.
\end{enumerate}
\end{theorem}

\begin{proof} Let $(F_\alpha)_{\alpha<\mathfrak c}$ be an enumeration of all uncountable closed subsets of $\IR^n$ by ordinal $<\mathfrak c$. Fix any point $p\in\IR^n\setminus\{\mathbf 0\}$. By transfinite induction, for every ordinal $\alpha<\mathfrak c$ we shall choose a point $z_\alpha\in F_\alpha$ such that the subgroup $G_\alpha\subset\IR^n$ generated by the set $\{z_\beta\}_{\beta<\alpha}$ does not contain the point $p$.  Assume that for some ordinal $\alpha<\mathfrak c$ we have chosen points $z_\beta$, $\beta<\alpha$, so that the subgroup $G_{<\alpha}$ generated by the set $\{z_\beta\}_{\beta<\alpha}$ does not contain $p$.
Consider the set $Z=\{\frac1n(p-g):n\in\IZ\setminus\{0\},\;g\in G_{<\alpha}\}$ and observe that it has cardinality $|Z|\le \w\cdot |G_{<\alpha}|\le \w+|\alpha|<\mathfrak c$. Since the uncountable closed subset $F_\alpha$ of $\IR^n$ has cardinality $|F_\alpha|=\mathfrak c$ (see \cite[6.5]{Ke}), there is a point $z_\alpha\in F_\alpha\setminus Z$. For this point we get $p\ne n z_\alpha+g$ for any $n\in\IZ\setminus\{0\}$, and $g\in G_{<\alpha}$. Consequently, the subgroup $G_{\alpha}=\{nz_\alpha+g:n\in\IZ,\;g\in G_{<\alpha}\}$ generated by the set $\{z_\beta\}_{\beta\le\alpha}$ does not contain the point $p$.
This completes the inductive step.

After completing the inductive construction, consider the subgroup $G$ generated by the set $\{a_\alpha\}_{\alpha<\mathfrak c}$ and observe that $p\notin G$ and $G$ meets every uncountable closed subset $F$ of $\IR^n$. Moreover, since $G$ meets the closed uncountable set $F-p$, the coset $p+G\subset \IR^n\setminus G$ meets $F$. So, both the subgroup $G$ and its complement $\IR^n\setminus G$ meet each uncountable closed subset of $\IR^n$, which means that $G$ is a Bernstein set in $\IR^n$.
The following proposition implies that the group $G$ has properties (2)--(4).
\end{proof}

\begin{proposition} Let $n\ge 2$. Every Bernstein subset $B$ of $\IR^n$ has the following properties:
\begin{enumerate}
\item $B$ is not Borel;
\item $B$ is connected and locally connected;
\item $B$ has dimension $\dim(B)=n-1$.
\end{enumerate}
\end{proposition}

\begin{proof} 1. By \cite[8.24]{Ke}, the Bernstein set $B$ is not Borel (more precisely, $B$ does not have the Baire property in $\IR^n$).
\smallskip

2. To prove that $B$ is connected and locally connected, it suffices to prove that for every open subset $U\subset\IR^n$ homeomorphic to $\IR^n$ the intersection $U\cap G$ is connected. Assuming the opposite, we could find two non-empty open disjoint sets $U_1,U_2\subset U$ such that $U\cap B=(U_1\cap B)\cup (U_2\cap B)$. Consider the complement $F=U\setminus(U_1\cup U_2)\subset U\setminus B$ and observe that $F$ is closed in $U$ and hence of type $F_\sigma$ in $\IR^n$. If $F$ is uncountable, then $F$ contains an uncountable closed subset of $\IR^n$ and hence meets the set $B$, which is not the case. So, the closed subset $F$ of $U$ is at most countable and separates the space $U\cong \IR^n$, which contradicts Theorem 1.8.14 of \cite{En}.
\smallskip

3. Since the subset $B$ has empty interior in $\IR^n$, we can apply Theorem 1.8.11 of \cite{En} and conclude that $\dim(B)<n$. On the other hand, Lemma 1.8.16 \cite{En} guarantees that $B$ has dimension $\dim(B)\ge n-1$ (since $B$ meets every non-trivial compact connected subset of $\IR^n$). So, $\dim(B)=n-1$.
\end{proof}

\end{document}